\documentclass[11pt,a4paper]{article}
\usepackage[left=3cm, right=3cm]{geometry}
\usepackage{graphicx}
\usepackage{mathtools}
\usepackage{amssymb}
\usepackage{amsthm}
\usepackage{thmtools}
\usepackage{xcolor}
\usepackage{nameref}
\usepackage{amsfonts}
\usepackage{hyperref}

\newtheorem{thm}{Theorem}
\newtheorem{lem}{Lemma}

\title{On the Precise Asymptotics of $ex(n,n,n,K_{2,t})$ for even $t$}
\author{Zilin Luo \\ wolfsonluo@gmail.com}
\date{}

\begin{document}
	\maketitle
	\begin{abstract}
		Let $K_{2,t}$ denote the complete bipartite graph. For an integer $n\ge 1$, let $ex(n,n,n,K_{2,t})$
		be the maximum number of edges in an $n\times n\times n$ tripartite graph (that is, a 3-partite graph with three parts each of size $n$) containing no copy of $K_{2,t}$. In this paper we prove that, for even $t\ge 2$,
		$$
		ex(n,n,n,K_{2,t}) \ge \frac{3\sqrt{t-1}}{\sqrt{2}}\, n^{3/2} + o(n^{3/2}).
		$$
		Combining our construction with earlier work of Tait and Timmons, we obtain
		$$
		\lim\limits_{n\to\infty} \frac{ex(n,n,n,K_{2,t})}{n^{3/2}} = \frac{3\sqrt{t-1}}{\sqrt{2}}, \qquad\text{for integer } t\ge 2.
		$$
	\end{abstract}
	\section{Introduction}
	
	For a graph $F$, the Turán number $ex(N,F)$ denotes the maximum number of edges in an $N$-vertex graph that does not contain $F$ as a subgraph. Determining $ex(N,F)$ is a fundamental problem in extremal graph theory. Classical results include Mantel's theorem \cite{M} (the case $F=K_3$) and Turán's theorem for complete graphs\cite{T}; both give exact values of $ex(N,F)$ for all $N$ in their respective settings. However, for a general graph $F$ the exact determination of $ex(N,F)$ is often difficult, and researchers thus frequently investigate its asymptotic behaviour.
	
	The Erd\H{o}s–Stone–Simonovits theorem determines the leading asymptotic term of $ex(N,F)$ in terms of the chromatic number $\chi(F)$ \cite{ES,ES2}
	\[
	ex(N,F) = \left(1 - \frac{1}{\chi(F)-1} + o(1)\right)\binom{N}{2}.
	\]
	This result effectively resolves the problem for non-bipartite graphs, but it yields only the weak bound $ex(N,F)=o(N^2)$ when $F$ is bipartite (since then $\chi(F)=2$). For bipartite $F$ much sharper asymptotics have been established in special cases. For example, Brown \cite{Brown} and, independently, Erd\H{o}s--Rényi--Sós\cite{ER,ERS} showed that for certain values of $N$,
	\[
	ex(N,C_4)=\tfrac12 N^{3/2}+o(N^{3/2}).
	\]
	Using finite-geometric constructions and building on the K\H{o}vári--Sós--Turán bound, F\"{u}redi \cite{Furedi} proved that for every integer $t\ge 1$,
	\[
	ex(N,K_{2,t+1}) = \frac{\sqrt{t}}{2}\, N^{3/2} + o(N^{3/2}).
	\]
	
	It is also natural to study bipartite Turán numbers with prescribed part sizes. For a bipartite graph $F$, let $ex(n,m,F)$ denote the maximum number of edges in an $F$-free bipartite graph with parts of sizes $n$ and $m$. In particular, the works in \cite{Furedi,KST} imply that
	\[
	ex(n,n,K_{2,t+1}) = \sqrt{t}\, n^{3/2} + o(n^{3/2}).
	\]
	For further results on bipartite Turán numbers the reader may consult \cite{3,4,5,6,7,8}.
	
	Recently, Tait and Timmons \cite{TT} studied extremal problems under chromatic constraints. Writing $ex_{\chi\le k}(N,F)$ for the maximum number of edges in an $F$-free graph on $N$ vertices with chromatic number at most $k$, they proved the upper bound
	\[
	ex_{\chi\le 3}(N,K_{s,t}) \le \left(\frac{1}{3}\right)^{1-1/s}\left(\frac{t-1}{2} + o(1)\right)^{1/s} N^{2-1/s},
	\]
	and showed this bound to be asymptotically tight in certain cases. In particular, when $s=2$ and $t$ is odd they obtain
	\[
	ex_{\chi\le 3}(N,K_{2,2r+1}) = \sqrt{\frac{r}{3}}\, N^{3/2} + o(N^{3/2}).
	\]
	Lv et al.\cite{LL} treated the case $s=2$, $t=2$ and proved
	\[
	ex_{\chi\le 3}(N,K_{2,2}) = \frac{1}{\sqrt{6}}\, N^{3/2} + o(N^{3/2}).
	\]
	
	In this paper we extend the construction and methods of Lv et al.\ to handle the remaining case $s=2$ with even $t$. Our main result establishes the asymptotic value of $ex_{\chi\le 3}(N,K_{2,t})$ (equivalently, of $ex(n,n,n,K_{2,t})$ after the appropriate change of parameters) for integer $t\ge 2$; see Theorem~\ref{thm:main} below for a precise statement.
	\section{Results}

	By definition we have
	\[
	ex(n,n,n,F) \le ex(3n,F).
	\]
	
	The following theorem gives an upper bound for $ex(n,n,n,F)$. The proof is essentially the same as the proof of Theorem~3 in \cite{LL}; for completeness we include a proof here.
	
	\begin{thm}\label{thm:upperbound}
		For every integer $n\ge 1$,
		\[
		ex(n,n,n,K_{2,t}) \leq \frac{3}{2}\, n\Big(1+\sqrt{2(t-1)(n-1)+1}\Big).
		\]
	\end{thm}

		\begin{proof}
			Let \(G\) be a 3-partite \(K_{2,t}\)-free graph with vertex classes \(V_1\cup V_2\cup V_3\), where \(|V_i|=n\) for \(i=1,2,3\). For \(1\le i\le 3\) (indices taken modulo \(3\)) define
			$m_{i+2} = E(V_i\cup V_{i+1})$. Assume vertices of $V_i$ are $v_1^i, v_2^i, \cdots, v_n^i$. Define
			$
			x_i=|N(v_i^1)\cap V_3|,\qquad y_i=|N(v_i^2)\cap V_3|,\quad 1\le i\le n,
			$
			so that \(m_2=\sum_{i=1}^n x_i\) and \(m_1=\sum_{i=1}^n y_i\).
			
			Since \(G\) is \(K_{2,t}\)-free, any two vertices in \(V_3\) have at most \(t-1\) common neighbours in  \(V_1\cup V_2\). Hence
			\[
			\sum_{i=1}^n\binom{x_i}{2}+\sum_{i=1}^n\binom{y_i}{2}\le (t-1)\binom{n}{2}.
			\]
			By convexity of the function \(\binom{x}{2}=\tfrac{x(x-1)}{2}\),
			\[
			2n\binom{\frac{m_1+m_2}{2n}}{2}
			\le \sum_{i=1}^n\binom{x_i}{2}+\sum_{i=1}^n\binom{y_i}{2}
			\le (t-1)\binom{n}{2}.
			\]
			Let \(s=m_1+m_2\). The above inequality leads to
			\[
			s^2-2ns-2n^2(t-1)(n-1)\le0,
			\]
			then
			\[
			s\le n\Bigl(1+\sqrt{2(t-1)(n-1)+1}\Bigr).
			\]
			By the same argument we obtain the identical upper bound for \(m_1+m_3\) and \(m_2+m_3\). Summing the three bounds and dividing by \(2\) gives
			\[
			e(G)=m_1+m_2+m_3
			\le \tfrac{3}{2}n\Bigl(1+\sqrt{2(t-1)(n-1)+1}\Bigr).
			\]
			Since this holds for every 3-partite \(K_{2,t}\)-free graph with parts of size \(n\), we conclude
			\[
			\operatorname{ex}(n,n,n,K_{2,t})\le \tfrac{3}{2}n\Bigl(1+\sqrt{2(t-1)(n-1)+1}\Bigr),
			\]
			completing the proof.
		\end{proof}

	Next we present a construction that yields a lower bound for $ex(n,n,n,K_{2,t})$. The case of odd $t$ was treated in \cite{TT}; below we consider the remaining case when $t$ is even. The construction below is inspired by the algebraic constructions of F\"{u}redi and of Lv et~al.
	
	Let $p$ be a prime with $(t-1)\mid (p-1)$, and write
	\[
	p-1 = a(t-1),
	\]
	then $a$ is even for even $t$. Define
	\[
	n_t=\frac{p(p-1)}{2(t-1)}=\frac{a}{2}\,p.
	\]
	Denote by $\mathbb{F}_p$ the finite field with $p$ elements, and by $\mathbb{F}_p^*$ its multiplicative group. Fix a primitive root $g$ generating $\mathbb{F}_p^*$.  Set
	\[
	B=\{g,g^2,\dots,g^{a/2}\},\qquad W=\{a,2a,\dots,(t-1)a\}.
	\]
	To prove the correctness of the construction we need the following lemma.
	
	\begin{lem}\label{lemma:unique}
		Fix any $f\in\mathbb{F}_p^*$. Define
		\[
		P_f=\{(w,b)\in W\times B \mid g^w b = f\},
		\qquad
		N_f=\{(w,b)\in W\times B \mid g^w b = -f\}.
		\]
		Then for every $f$ we have
		\[
		\{|P_f|,|N_f|\} = \{0,1\},
		\]
		this yields that each of $P_f$ and $N_f$ contains at most one element.
	\end{lem}
	
	\begin{proof}
		Since $g$ is a primitive root modulo $p$ we have $g^{(p-1)/2}\equiv -1\pmod p$. Using $p-1=a(t-1)$ this implies
		\[
		g^{\frac{(t-1)a}{2}}=-1,\qquad \text{in }\mathbb{F}_p.
		\]
		
		Write $f$ uniquely in the form
		$
		f = g^{c a + d},
		$
		where $c\in\{1,2,\dots,t-1\}$ and $d\in\{1,2,\dots,a\}$. Note that any element of $W$ has the form $w=ja$ for some $j\in\{1,2,\dots,t-1\}$, and any element of $B$ is $g^e$ for some $e\in\{1,\dots,a/2\}$.
		
		First suppose $1\le d\le a/2$. Taking $w=ca\in W$ and $b=g^d\in B$ we get $g^w b = g^{ca+d}=f$, hence $(ca,g^d)\in P_f$. If there were another pair $(w',b')\in P_f$ with $w'=j'a$ and $b'=g^{e'}$, then
		\[
		j'a + e' \equiv ca + d \pmod{a(t-1)},
		\]
		and reducing both sides modulo $a$ gives $e'\equiv d\pmod a$. Since $e'\in\{1,\dots,a/2\}$ and $d\in\{1,\dots,a/2\}$, we must have $e'=d$, and consequently $j'=c$. Thus $(ca,g^d)$ is the unique element of $P_f$, and $N_f=\varnothing$ in this case.
		
		Next suppose $a/2 +1 \leq  d \le a$, Using $g^{\frac{(t-1)a}{2}}=-1$ we obtain
		\[
		-f = g^{\frac{(t-1)a}{2}} g^{ca+d} = g^{\big(c+\tfrac{t}{2}\big)a + d - \tfrac{a}{2}}.
		\]
		Similarly argument shows that $P_f = 0$ and, $N_f = \{(ca+\frac{ta}{2}, g^{d-\frac{a}{2}})\}$ if $c\leq \frac{t}{2}-1$ and $N_f = \{(ca-\frac{ta}{2} + a, g^{d-\frac{a}{2}})\}$ if $c\geq \frac{t}{2}$. 
		
		Thus in every case each of $P_f$ and $N_f$ contains at most one element, proving the lemma.
	\end{proof}
	
	 The next Theorem gives a lower bound of $ex_{\chi\leq 3}(n_t, n_t, n_t, K_{2,t})$.
	 
	 \begin{thm} \label{thm:lowerbound}
	 	Let $t\geq 2$ be an even integer, $p\geq 3$ is  a prime and $p \equiv 1 \pmod{t-1}$, $n_t = \frac{p(p-1)}{2(t-1)}$. Then we have
	 	$$
	 	ex(n_t,n_t,n_t,K_{2,t}) \geq \frac{3p(p-1)^2}{4(t-1)}.
	 	$$
	 \end{thm}
\begin{proof}
	Let
	\[
	V = \{(b,x) \mid b\in B,\ x \in \mathbb{F}_p\}.
	\]
	The graph $G$ is tripartite, and its three vertex classes $V_1,V_2,V_3$ are copies of $V$. For $i=1,2,3$ (indices taken modulo $3$), a vertex $(b,x)\in V_i$ is adjacent to $(c,y)\in V_{i+1}$ if and only if there exists $w\in W$ such that
	\[
	bc = g^w(x-y)\quad\text{in }\mathbb{F}_p.
	\]
	For fixed $b,c,w,x$ the equation determines a unique $y$, hence $G$ has
	\[
	3\times \frac{a}{2} \times \frac{a}{2} \times (t-1) \times p = \frac{3p(p-1)^2}{4(t-1)}
	\]
	edges.
	
	We now show that $G$ is $K_{2,t}$-free. It suffices to prove that for any two distinct vertices $u=(b_1,x_1)$ and $v=(b_2,x_2)$ we have $|N(u)\cap N(v)|\le t-1$. We treat two cases.
	
	\begin{itemize}
		\item \textbf{Case 1.} $u,v$ belong to the same part. Without loss of generality assume $u,v\in V_1$. Let $s=(c,y)\not\in V_1$ be a common neighbour of $u$ and $v$. If $s\in V_2$, then there exist $w_1,w_2\in W$ with
		$
		b_i c = g^{w_i}(x_i-y),\quad i=1,2,
		$
		and hence
		$
		(g^{-w_1} b_1 - g^{-w_2}b_2)c = x_1-x_2,
		$
		i.e.
		\begin{eqnarray}\label{equa1}
			g^{-w_1} (b_1-g^{w_1-w_2}b_2)c = x_1-x_2.
		\end{eqnarray}
		
		If $s\in V_3$, then there exist $w_1,w_2\in W$ with
		$
		b_i c = g^{w_i}(y-x_i),\quad i=1,2,
		$
		and hence
		$
		(g^{-w_1} b_1 - g^{-w_2}b_2)c = x_2-x_1,
		$
		i.e.
		\begin{eqnarray}\label{equa2}
			g^{-w_1} (b_1-g^{w_1-w_2}b_2)c = -(x_1-x_2).
		\end{eqnarray}
		
		If $g^{w_1}a_1- g^{w_2}a_2=0$, then by Lemma~\ref{lemma:unique} we would have $w_1=w_2$ and $a_1=a_2$, which forces $x_1=x_2$, contradicting $u\neq v$. Thus $g^{w_1}a_1-g^{w_2}a_2\ne 0$. Fixing $g^{w_1-w_2}=g^z$, set
		$
		f = (b_1-g^{z}b_2)^{-1}(x_1-x_2).
		$
		Then \eqref{equa1} becomes $g^{-w_1}c = f$ and \eqref{equa2} becomes $g^{-w_1}c=-f$.
		By Lemma~\ref{lemma:unique}, among the two equations exactly one has a unique solution $(w_1,c)$ while the other has none, hence for each chosen value of $g^z$ there is at most one corresponding common neighbour $s$. Since $g^z$ can take at most $t-1$ distinct values, it follows that there are at most $t-1$ common neighbours in this case, i.e. $|N(u)\cap N(v)|\le t-1$.
		
		\item \textbf{Case 2.} $u,v$ lie in different parts. Without loss of generality assume $u\in V_1$ and $v\in V_2$. Any common neighbour $s=(b,y)$ must lie in $V_3$. Then there exist $w_1,w_2\in W$ with
		$
		b_1c = g^{w_1}(y-x_1),\qquad b_2c = g^{w_2}(x_2-y),
		$
		and hence
		$
		(g^{-w_1}b_1+g^{-w_2}b_2)c = x_2-x_1.
		$
		By Lemma~\ref{lemma:unique} the coefficient $g^{-w_1}b_1+g^{-w_2}b_2\ne 0$. Fix $g^{w_1-w_2}=g^z$ and set
		$
		f = (b_1+g^{z}b_2)^{-1}(x_2-x_1).
		$
		Again by Lemma~\ref{lemma:unique} the equation $g^{-w_1}c = f$ has at most one solution, so for each choice of $g^z$ there is at most one common neighbour $s$. As $g^z$ has at most $t-1$ values, we obtain $|N(u)\cap N(v)|\le t-1$ in this case as well.
	\end{itemize}
	
	Therefore $G$ is $K_{2,t}-$ free, hence
	$$
		 	ex(n_t,n_t,n_t,K_{2,t}) \geq \frac{3p(p-1)^2}{4(t-1)}.
	$$
\end{proof}
Since for every positive integer $m$ there exists a prime $p$ such that $(t-1) \mid (p-1)$ with $m\le p\le(1+o(1))m$, combining Theorem~\ref{thm:lowerbound} and Theorem~\ref{thm:upperbound} we obtain that for every positive even integer $t$,
\[
ex(n,n,n,K_{2,t}) = 3\sqrt{\frac{t-1}{2}}\, n^{3/2} + o(n^{3/2}).
\]
Together with the results of Tait and Timmons for odd $t$\cite{TT}, this yields the following theorem.

\begin{thm}\label{thm:main}
	For every integer $t\ge 2$,
	\[
	ex_{\chi\le 3}(n,K_{2,t}) = \sqrt{\frac{t-1}{6}}\, n^{3/2} + o(n^{3/2}).
	\]
\end{thm}

\section{Conclusion}
In this work, by using algebraic constructions plus standard extremal tools, we determined the main asymptotic constants for $K_{2,t}$ in the three-colourable/tripartite settings. The main outcome is (Theorem~\ref{thm:main})
\[
ex_{\chi\le 3}(n,K_{2,t}) = \sqrt{\frac{t-1}{6}}\,n^{3/2} + o(n^{3/2})\qquad (t\ge 2).
\]

In future work I will continue along this line of research. For example, I plan to study the extremal function
$
ex_{\chi\le k}(n,F)
$
for larger values of \(k\), and to investigate other forbidden subgraphs \(F\) such as cycles \(C_\ell\). In particular, it would be interesting to compare even and odd cycles under tripartite constraints and to seek matching constructions and upper bounds.

\section*{ACKNOWLEDGMENTS}
I would like to express my gratitude to my family for their unwavering support throughout this study.

\end{document}